\documentclass[12pt, reqno]{amsart}
\usepackage{amssymb,amscd,amsmath}
\usepackage[all]{xy}
\usepackage{graphicx}
\usepackage{color}
\usepackage{enumitem}

\usepackage[margin=2.6cm]{geometry}

\newtheorem{theorem}[subsection]{Theorem}
\newtheorem{lemma}[subsection]{Lemma}
\newtheorem{sublemma}[subsubsection]{Lemma}

\newtheorem{corollary}[subsection]{Corollary}

\newtheorem{definition}[subsection]{Definition}

\newtheorem{remark}[subsection]{Remark}

\makeatletter
\@addtoreset{subsection}{section}
\@addtoreset{equation}{section}
\@addtoreset{figure}{section}
\@addtoreset{table}{section}
\makeatother

\makeatletter
\newcommand\testshape{family=\f@family; series=\f@series; shape=\f@shape.}
\def\myemphInternal#1{\if n\f@shape%
\begingroup\itshape #1\endgroup\/%
\else\begingroup\bfseries #1\endgroup%
\fi}
\def\myemph{\futurelet\testchar\MaybeOptArgmyemph}
\def\MaybeOptArgmyemph{\ifx[\testchar \let\next\OptArgmyemph
                 \else \let\next\NoOptArgmyemph \fi \next}
\def\OptArgmyemph[#1]#2{\index{#1}\myemphInternal{#2}}
\def\NoOptArgmyemph#1{\myemphInternal{#1}}
\makeatother

\author{Sergiy Maksymenko, Eugene Polulyakh}
\title{Foliations with non-compact leaves on surfaces}
\address{Topology Department, Institute of Mathematics, Ukrainian National Academy of Science, Te\-re\-shchen\-kiv\-ska str. 3, 01601 Kyiv, Ukraine}
\email{maks@imath.kiev.ua, polulyah@imath.kiev.ua}



\begin{document}

\begin{abstract}
We study non-compact surfaces obtained by gluing strips $\mathbb{R}\times(-1,1)$ with at most countably many boundary intervals along some these intervals.
Every such strip possesses a foliation by parallel lines, which gives a foliation on the resulting surface.
It is proved that the identity path component of the group of homeomorphisms of that foliation is contractible.
\end{abstract}

\maketitle

\newcommand\RRR{\mathbb{R}}
\newcommand\CCC{\mathbb{C}}
\newcommand\ZZZ{\mathbb{Z}}
\newcommand\NNN{\mathbb{N}}
\newcommand\FFF{\mathbb{F}}

\newcommand\EE{\mathcal{E}}
\newcommand\FF{\mathcal{F}}
\newcommand\LL{\mathcal{L}}
\newcommand\QQ{\mathcal{Q}}
\newcommand\XX{\mathcal{X}}
\newcommand\YY{\mathcal{Y}}

\newcommand\Fix{\mathrm{Fix}}
\newcommand\Int{\mathrm{Int}}
\newcommand\Fr{\mathrm{Fr}}
\newcommand\Per{\mathrm{Per}}
\newcommand\im{\mathrm{im\,}}
\newcommand\id{\mathrm{id}}
\newcommand\cl[1]{\overline{#1}}
\newcommand\spec{\mathrm{spec}}
\newcommand\diam{\mathrm{diam}}
\newcommand\defeq{:=}

\newcommand{\strip}{S}
\newcommand{\stripSurf}{Z}
\newcommand{\bdX}{X}
\newcommand{\bdY}{Y}
\newcommand{\func}{f}
\newcommand{\hstrip}{\widehat{S}}

\newcommand{\stripS}{S}
\newcommand{\stripT}{T}

\newcommand\PS{\mathcal{P}(\strip)}
\newcommand\PZ{\mathcal{P}(\stripSurf)}

\newcommand\Homeo{\mathcal{H}}
\newcommand\HRf{\Homeo(\RRR,\func)}
\newcommand\HZ{\Homeo(\stripSurf)}
\newcommand\HR{\Homeo(\RRR)}
\newcommand\dif{h}
\newcommand\gdif{g}
\newcommand\GG{\mathcal{G}}
\newcommand\CZ{C(\stripSurf)}     

\newcommand\Stab{\mathcal{S}}     
\newcommand\Stabf{\Stab(\func)}   
\newcommand\StabIdf{\Stab_0(\func)}   
\newcommand\Rf{\mathcal{R}(\func)}
\newcommand\RIdf{\mathcal{R}_0(\func)}

\newcommand\HF[1]{\mathcal{H}\bigl(\mathcal{F}_{#1}\bigr)}
\newcommand\HOF[1]{\mathcal{H}_{0}\bigl(\mathcal{F}_{#1}\bigr)}
\newcommand\HOFpr[1]{\mathcal{H}_{0}\bigl(\mathcal{F}_{#1}\bigr)'}

\newcommand\FCanon[1]{\mathcal{F}_{#1}}
\newcommand\FS{\FCanon{\strip}}
\newcommand\FZ{\FCanon{\stripSurf}}
\newcommand\FZpr{\FCanon{\stripSurf'}}

\newcommand\HFZ{\mathcal{H}\bigl(\mathcal{F}_{\stripSurf}\bigr)}
\newcommand\HOFZ{\mathcal{H}_{0}\bigl(\mathcal{F}_{\stripSurf}\bigr)}
\newcommand\HOFZpr{\mathcal{H}_{0}\bigl(\mathcal{F}_{\stripSurf}\bigr)'}

\newcommand\KRS{\Gamma_{\strip}}
\newcommand\KRZ{\Gamma_{\stripSurf}}

\newcommand\KRZT{\KRZ'}
\newcommand\Aut{\mathrm{Aut}}

\newcommand\xcoord{\lambda}
\newcommand\ycoord{\mu}

\newcommand\sat{\mathrm{Sat}}
\newcommand\qmap{q}

\newcommand\Aind{\mathbf{A}}
\newcommand\Bind{\mathbf{B}}
\newcommand\aind{{\alpha}}
\newcommand\bind{{\beta}}
\newcommand\cind{{\gamma}}

\newcommand\typeInternal{{\rm(a)}}
\newcommand\typeBd{{\rm(b)}}
\newcommand\typeGlued{{\rm(c)}}
\newcommand\typeCycle{{\rm(c1)}}
\newcommand\typeReduce{{\rm(c2)}}
\newcommand\typeSpec{{\rm(c3)}}

\section{Introduction}
The qualitative part of one complex variable function theory concerns with topological classification of analytical and pseudoharmonic functions as well as with foliations of their level-sets.
Such kind of problems was considered by S.~Stoilov~\cite{Stoilov:1964} and G.~T.~Whyburn~\cite{Weaver:AMSCP:1942} who introduced notions of internal and light open maps (respectively) which reflect certain essential topological features of analytical mappings.
At the same time foliations by level sets of harmonic function on the plane were studied by W.~Kaplan~\cite{Kaplan:DJM:1940}.

We will say that a continuous function $f:\stripSurf\to\RRR$ \myemph{agrees} with a $1$-dimensional foliation $\FF$ on $\stripSurf$ if 
\begin{itemize}
\item
each leaf of $\FF$ is a connected component of some level-set $f^{-1}(c)$, $c\in\RRR$;
\item
for each $z\in\RRR^2$ there are local coordinates $(u,v)$ in which $z=(0,0)$ and $f(u,v) = u + \mathrm{const}$.
\end{itemize}

Suppose $\FF$ is a one-dimensional foliation on $\RRR^2$ with all leaves non-compact.
W.~Kaplan~\cite{Kaplan:DJM:1940}, \cite{Kaplan:DJM:1941} extending old result by E.~Kamke~\cite{Kamke:MA:1928} proved that then there exists a continuous function $f:\RRR^2\to\RRR$ which agrees with $\FF$.
Moreover, one can find at most countable covering $\{\strip_{\aind}\}_{\aind\in\Aind}$ of $\RRR^2$ such that
\begin{itemize}
\item[(1)]
each $\strip_{\aind}$ consists of entire leaves of $\FF$;
\item[(2)]
the foliation on $\strip_{\aind}$ is equivalent to the foliation on the plane $\RRR^2$ or on the half-plane by parallel lines.
\end{itemize}
In other words, $\RRR^2$ is glued of countably many strips along open boundary intervals, see Figure~\ref{fig:strips_example}.
\begin{figure}[h]
\includegraphics[height=3cm]{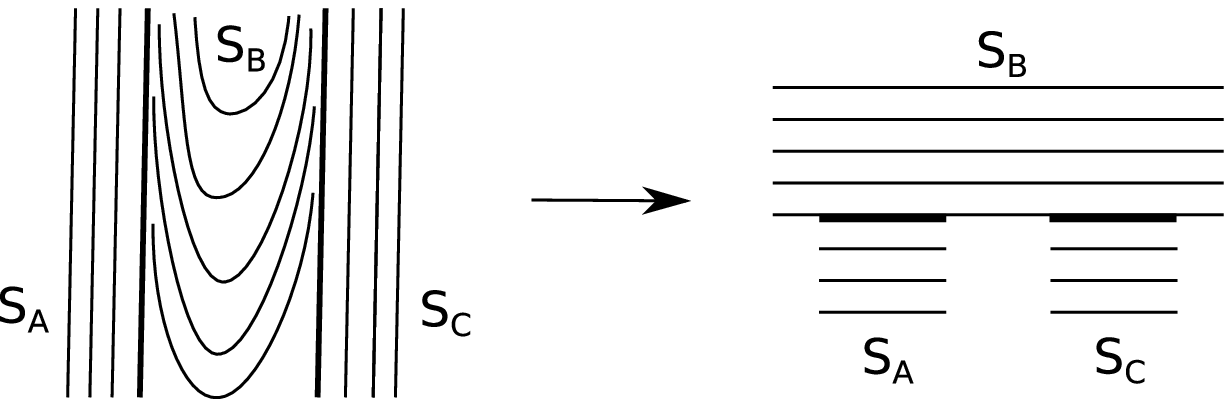}
\caption{}\label{fig:strips_example}
\end{figure}

In~\cite{Kaplan:TrAMS:1948} W.~Kaplan also shown that there exists a homeomorphism $\dif:\RRR^2\to\RRR^2$ such that the function $f\circ h$ is harmonic.
This result was further extended to foliations with singularities W.~Boothby~\cite{Boothby:AJM_1:1951}, \cite{Boothby:AJM_2:1951}, M.~Morse and J.~Jenkins~\cite{JenkinsMorse:AJM:1952}, M.~Morse~\cite{Morse:FM:1952}.
See also~\cite{JenkinsMorse:AJM:1952}, \cite{JenkinsMorse:APNASUSA:1953}, \cite{JenkinsMorse:ActaMath:1954}, \cite{JenkinsMorse:UMichPress:1955}, \cite{Morse:JMPA:1956}, \cite{SharkoSoroka:MFAT:2015}

In the last twenty years the interest to the topological classification of functions on surfaces arises due to a progress in the theory of Hamiltonial dynamical systems of small degrees of freedom, see e.g. A.~Fomenko and A.~Bolsinov~\cite{BolsinovFomenko:1997}, A.~Oshemkov~\cite{Oshemkov:PSIM:1995}, V.~Sharko ~\cite{Sharko:UMZ:2003}, \cite{Sharko:Zb:2006}, E.~Polulyakh and I.~Yurchuk~\cite{PolulyakhYurchuk:ProcIM:2009}, E.~Polulyakh~\cite{Polulyakh:UMZ:ENG:2015}.

\medskip

In the present paper we will study homotopical properties of foliations having properties (1) and (2) above on arbitrary open surfaces $\stripSurf$.
Thus such a surface is obtained from a family of strips $\{\strip_{\aind}\}_{\aind\in\Aind}$ glued along some boundary intervals, and therefore such surfaces will be called \myemph{stripped}.
Every strip admits a natural foliation by parallel lines which gives a foliation $\FZ$ on $\stripSurf$ with all leaves non-compact.
We will call this foliation \myemph{canonical}.
Let $\HFZ$ be the group of all homeomorphisms of $\stripSurf$ which maps leaves of $\FZ$ onto leaves of $\FZ$, and $\HOFZ$ be the identity path component of $\HFZ$ with respect to the open compact topology.
We will prove (Theorem~\ref{th:characterzation_of_H0}) that $\HOFZ$ is contractible.
Hence the homotopy type of $\HFZ$ reduces to the computation of the \myemph{homeotopy group} (or \myemph{mapping class group}) $\pi_0\HFZ= \HFZ/\HOFZ$ of the foliation $\FZ$.
As an example we characterize stripped surfaces consisting of one strip (Theorem~\ref{th:reducing_stripped_surfaces}).

\section{Stripped surfaces}\label{sect:stripped_surfaces}

\begin{definition}\label{def:model_strip}
A subset $\strip \subset \RRR\times[-1,1]$ will be called a \myemph{model strip} if 
\begin{enumerate}
\item[\rm(1)]
$\RRR\times (-1,1) \ \subset \ \strip$,
\item[\rm(2)]
the intersection $\strip \cap (\RRR\times \{-1,1\})$ is a (possibly empty) union of open finite intervals with mutually disjoint closures.
\end{enumerate}
\end{definition}

For example, $\RRR\times(-1,1)$ is a model strip, while $\RRR\times[-1,1]$, $\RRR\times(-1,1]$, $\RRR\times[-1,1)$ are not.

For a model strip $\strip$ we will use the following notation:
\begin{align*}
\partial_{-}\strip &:= \strip \ \cap \ \RRR\times \{-1\}, &
\partial_{+}\strip &:= \strip \ \cap \ \RRR\times \{1\}, &
\partial\strip &:= \partial_{-}\strip \ \cup \ \partial_{+}\strip.
\end{align*}

Connected components of $\partial_{-}\strip$ (resp. $\partial_{+}\strip$) will be called \myemph{lower} (resp. \myemph{upper}) boundary intervals.

\begin{definition}\label{def:stripped_surface}
A \myemph{stripped surface} is the quotient space
\begin{equation}\label{equ:stripped_surface}
\stripSurf = \bigsqcup\limits_{\aind\in\Aind} \strip_{\aind} \left/ \{ \bdY_{\bind} \stackrel{\phi_{\bind}}{\sim} \bdX_{\bind} \}_{\bind\in \Bind} \right.,
\end{equation}
where 
\begin{itemize}
\item[\rm(a)]  $\bigsqcup\limits_{\aind\in\Aind} \strip_{\aind}$ is a disjoint union of model strips;

\medskip

\item[\rm(b)] 
$\{\bdX_{\bind}, \bdY_{\bind}\}_{\bind\in \Bind} \subset \bigcup\limits_{\aind\in\Aind} \partial\strip_{\aind}$
is a family of pairs of boundary intervals such that $\bdX_{\bind} \not= \bdY_{\bind}$, $\bdY_{\bind} \not = \bdY_{\bind'}$ and $\bdX_{\bind} \not = \bdX_{\bind'}$ for $\bind\not=\bind' \in \Bind$;

\medskip

\item[\rm(c)]
$\phi_{\bind}: \bdY_{\bind} \to  \bdX_{\bind}$, $\bind\in \Bind$, is an {\bfseries affine} homeomorphism preserving or reversing orientations.
\end{itemize}
\end{definition}

Thus a stripped surface is a surface obtained from a family model strips by identifying some pairs of boundary intervals via affine homeomorphisms.
It is allowed that two strips are glued along more than one pair of boundary components.
One may also glue together intervals belonging to the boundary of same strip $\strip$, and even to the same lower or upped part of $\partial\strip$.

\begin{remark}\rm
Notice that if $a<b$ and $c<d$, then there are exactly two \textit{affine} homeomorphisms $\phi^{+}, \phi^{-}: (a,b) \to (c,d)$ such that $\phi^{+}$ preserves orientation and $\phi^{-}$ reverses it.
Namely,
\begin{equation}\label{equ:affine_gluing}
\phi^{+}(t) = \frac{d-c}{b-a}\bigl(t - a \bigr) + c,
\qquad \qquad 
\phi^{-}(t) = \frac{c-d}{b-a}\bigl(t - a \bigr) + d,
\end{equation}
for $t\in(a,b)$.
\end{remark}

\begin{remark}\rm
The assumption that the gluing maps are affine is technical and not crucial, however it will be essentially used in the proof of Lemma~\ref{lm:HOFZpr_is_contractible}.
\end{remark}

Let $\stripSurf$ be a stripped surface, defined by~\eqref{equ:stripped_surface}, and 
\[ \qmap: \bigsqcup\limits_{\aind\in\Aind} \strip_{\aind}  \longrightarrow \stripSurf \]
be the quotient map.
Then $\stripSurf$ is a non-compact two-dimensional manifold which can be nonconnected, non-orientable, and have boundary.
Every connected component of $\partial\stripSurf$ is an interval.
Also notice that a subset $U\subset\stripSurf$ is open if and only if $\qmap^{-1}(U)$ is open in $\bigsqcup\limits_{\aind\in\Aind}\strip_{\aind}$.

For each $\aind\in \Aind$ let 
\[ \xi_{\aind} :\strip_{\aind} \hookrightarrow \bigsqcup\limits _{\aind\in\Aind} \strip_{\aind} \xrightarrow{~\qmap~} \stripSurf\] be the composition of the inclusion of $\strip_{\aind}$ into $\bigsqcup_{\aind\in\Aind} \strip_{\aind}$ with the quotient map $\qmap$.
We will call $\xi_{\aind}$ a \myemph{chart} map corresponding to $\strip_{\aind}$.

It will also be convenient to use the following notation:
\begin{align*}
\hstrip_{\aind} &= \xi_{\aind}(\strip_{\aind}), &
\partial_{-}\hstrip_{\aind} &= \xi_{\aind}(\partial_{-}\strip_{\aind}), &
\partial_{+}\hstrip_{\aind} &= \xi_{\aind}(\partial_{+}\strip_{\aind}), &
\partial\hstrip_{\aind} &= \xi_{\aind}(\partial\strip_{\aind}).
\end{align*}
In particular, the image $\hstrip_{\aind}=\xi_{\aind}(\strip_{\aind})$ will be called a \myemph{strip} of $\stripSurf$.
Notice that if $\xi_{\aind}$ is not an embedding, then $\partial_{-}\hstrip_{\aind}$ and $\partial_{+}\hstrip_{\aind}$ may intersect.

On the other hand, the assumptions (b) and (c) guarantee that both restrictions $\xi_{\aind}|_{\partial_{-}\strip_{\aind}}$ and $\xi_{\aind}|_{\partial_{+}\strip_{\aind}}$ are injective.

\medskip

\section{Canonical foliation on a stripped surface}\label{sect:canon_foliation}
Notice that each model strip $\strip$ admits a one-dimensional foliation $\FS$ whose leaves are connected components of $\partial\strip$ and sets $\RRR\times\{t\}$, $t\in(-1,1)$.

More generally, let $\stripSurf$ be a stripped surface.
Since its strips are glued by homeomorphisms of leaves, the foliations on strips of $\stripSurf$ yield a foliation $\FZ$ on $\stripSurf$.
We will call this foliation \myemph{canonical}.

Let $\Gamma(\FF) = \stripSurf / \FF$ be the space of leaves endowed with the corresponding quotient topology, and $p:\stripSurf \to \Gamma(\FF)$ be the quotient map.
Then by definition a subset $V \subset \Gamma(\stripSurf)$ is open if and only if its inverse image $p^{-1}(V)$ is open in $\stripSurf$.
Thus open subsets of $\Gamma(\FF)$ can be regarded as open saturated subsets of $\stripSurf$.

\begin{lemma}
$\Gamma(\FF)$ is a $T_1$-space.
\end{lemma}
\begin{proof}
We should prove that every one-point set $\{x\} \subset \Gamma(\FF)$ is closed in $\Gamma(\FF)$, i.e. that each leaf $\omega$ of $\FF$ is closed in $\stripSurf$.
Since $\stripSurf = \bigsqcup_{\aind\in\Aind} \strip_{\aind} \left/ \{ \bdY_{\bind} \stackrel{\phi_{\bind}}{\sim} \bdX_{\bind} \}_{\bind\in \Bind} \right.$, see~\eqref{equ:stripped_surface}, we should check that every leaf in a model strip is closed.
But the latter is evident for leaves belonging to interiors of strips $\Int{\strip_{\aind}} = \RRR\times(-1,1)$ and follows from (2) of Definition~\ref{def:model_strip} for leaves belonging to $\partial\strip_{\aind} = \strip_{\aind} \setminus \Int{\strip_{\aind}}$.
\end{proof}

A homeomorphism $\dif:\stripSurf\to\stripSurf'$ between stripped surfaces will be called an \myemph{$\FF$-homeo\-morphism} whenever it maps leaves of $\FZ$ onto leaves of $\FZpr$.

Evidently, for each leaf $\omega$ we have exactly one of the following possibilities.

\smallskip

\begin{enumerate}[leftmargin=*]
\item[\typeInternal]
$\omega$ belongs to $\hstrip_{\cind} \setminus \partial\hstrip_{\cind}$ for some $\cind\in \Aind$; in this case $\omega$ will be called \myemph{internal}.
\smallskip

\item[\typeBd]
$\omega \subset \partial\hstrip_{\cind} \subset \partial\stripSurf$ for some $\cind\in \Aind$; in this case $\omega$ will be called a \myemph{boundary} leaf.
\smallskip

\item[\typeGlued]
$\omega \subset \partial_{\varepsilon}\hstrip_{\cind} \cap \partial_{\varepsilon'}\hstrip_{\cind'}$ for some $\cind,\cind'\in \Aind$ and $\varepsilon,\varepsilon'\in\{\pm\}$.
Then $\omega = \xi_{\cind}(\bdX_{\bind}) = \xi_{\cind'}(\bdY_{\bind})$ for some $\bind\in \Bind$.
This situation splits into the following three cases:
\smallskip

\begin{itemize}[leftmargin=1.5cm]
\item[\typeCycle]
$\cind = \cind'$, $\bdX=\partial_{\varepsilon}\strip_{\cind}$, and $\bdY=\partial_{\varepsilon'}\strip_{\cind}$;
so $\omega = \partial_{-}\hstrip_{\cind} = \partial_{+}\hstrip_{\cind} = \partial\hstrip_{\cind}$;

\smallskip

\item[\typeReduce]
$\cind \not= \cind'$, $\bdX=\partial_{\varepsilon}\strip_{\cind}$, and $\bdY=\partial_{\varepsilon'}\strip_{\cind}$, so
$\omega = \partial_{\varepsilon}\hstrip_{\cind} = \partial_{\varepsilon'}\hstrip_{\cind'}$;

\item[\typeSpec]
$\bdX\not=\partial_{\varepsilon}\strip_{\cind}$ or $\bdY\not=\partial_{\varepsilon'}\strip_{\cind}$,
in this case $\omega$ will be called \myemph{special}.

\smallskip

\end{itemize}
\end{enumerate}

Our aim is to reduce the situation to the case when there is no leaves of types {\typeCycle} and {\typeReduce}.
For this we need the following lemma.

\begin{lemma}\label{lm:reducing_strips}
Let $\stripS = \RRR\times(-1,1)$ and $\stripT = \stripS \setminus \bigl(  (-\infty, -1] \cup [1,+\infty) \bigr)\times0$.
Then there exists a homeomorphism $\dif:\stripS \to \stripT$ fixed outside $\RRR\times(-\tfrac{1}{2}, \tfrac{1}{2})$ and preserving horizontal lines, that is $\dif(x,y) = (\alpha(x,y), y)$ for some continuous function $\alpha$, see Figure~{\rm\ref{fig:reducing_c3}}.
\end{lemma}
\begin{proof}
Let $\sigma:\RRR\to(-1,1)$ be a $C^{\infty}$-diffeomorphism given by $\sigma(t) = \frac{t}{t^2+1}$.
\begin{figure}[h]
\includegraphics[height=1.3cm]{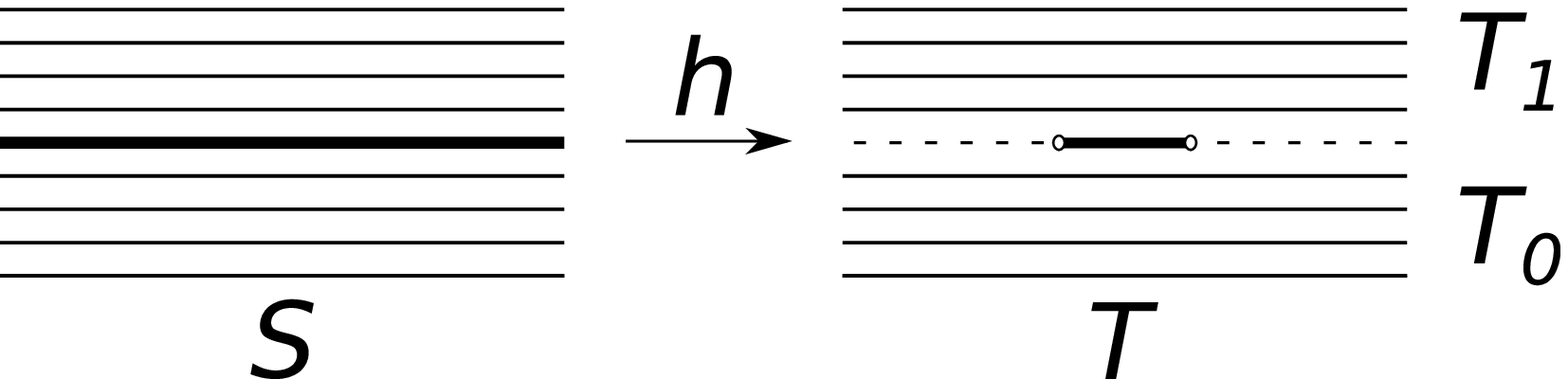}
\caption{}\label{fig:reducing_c3}
\end{figure}
Evidently, $\sigma'(t) = \frac{1-t^2}{(1+t^2)^2}$.
Define $\dif:\stripS \to \stripT$ by the formula:
\[
\dif(x,y) = 
\begin{cases}
\bigl(\sigma(x),0\bigr), & y=0, \\
\Bigl(\sigma'(-\tfrac{1}{|y|})\cdot (x + \tfrac{1}{|y|}) +  \sigma(-\tfrac{1}{|y|})\, , \ y \Bigr), & y\not=0, \ x \in (-\infty, -\tfrac{1}{|y|}], \\ 
\Bigl(\sigma(x),\, y\Bigr), & y\not=0, \ x \in [-\frac{1}{|y|}, \frac{1}{|y|}], \\ \medskip
\Bigl(\sigma'(\tfrac{1}{|y|})\cdot (x - \tfrac{1}{|y|}) +  \sigma(\tfrac{1}{|y|})\, , \ y \Bigr), & y\not=0, \ x \in (\tfrac{1}{|y|}, +\infty).
\end{cases}
\]
The verification that $\dif$ is indeed a homeomorphism having the corresponding properties we leave for the reader.
\end{proof}
\begin{corollary}\label{cor:reducing_infinite_strips}
Let $\stripS = \RRR^2$ and $\stripT = \RRR^2\setminus \bigl(  (-\infty, -1] \cup [1,+\infty) \bigr)\times \{2n\}_{n\in\ZZZ}$.
Then there exists a homeomorphism $\dif:\stripS \to \stripT$ fixed outside $\RRR\times \cup_{n\in\ZZZ} (2n-\tfrac{1}{2}, 2n+\tfrac{1}{2})$ and preserving horizontal lines, that is $\dif(x,y) = (\alpha(x,y), y)$ for some continuous function $\alpha$, see Figure~{\rm\ref{fig:reducing_c3_infinite}}.
\end{corollary}

\begin{figure}[!htbp]
\includegraphics[height=1.5cm]{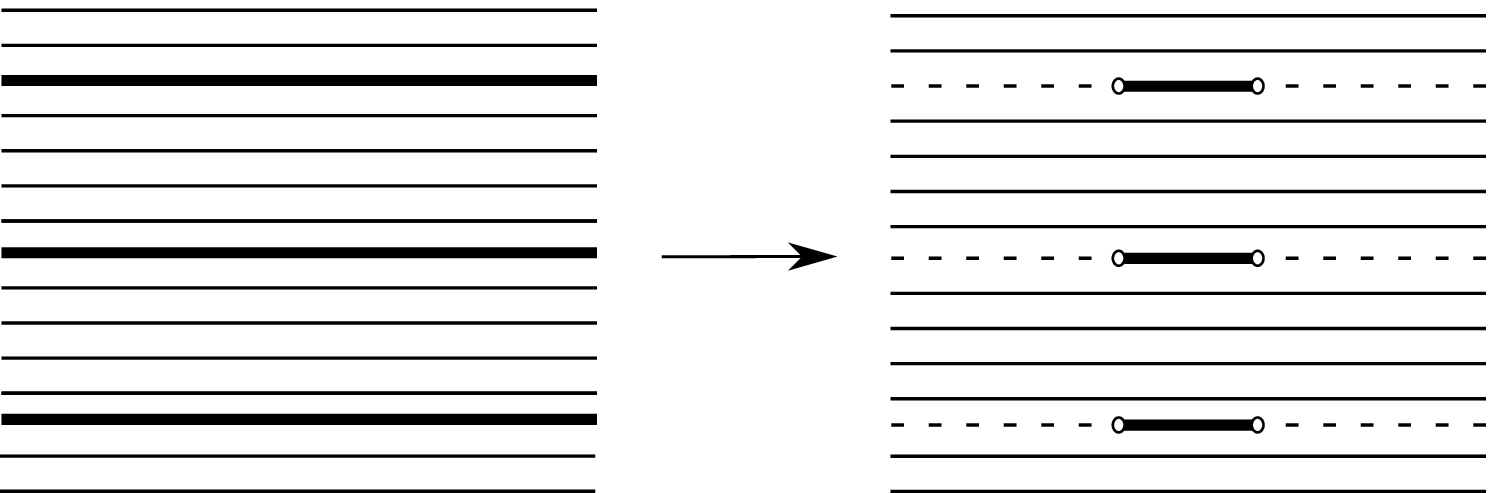}
\caption{}\label{fig:reducing_c3_infinite}
\end{figure}

\begin{remark}\label{rem:eliminatinc_c3_leaf}\rm
Notice that $\stripT$ from Lemma~\ref{lm:reducing_strips} is a stripped surface glued of two strips $\stripT_0$ and $\stripT_1$ so that $\omega=(-1,1)\times 0$ is a leaf of type {\typeReduce}, see Figure~\ref{fig:reducing_c3}.
Then Lemma~\ref{lm:reducing_strips} claims that these two strips can be replaced with a unique one so that the corresponding leaf of type {\typeReduce} becomes internal, i.e. having type {\typeInternal}.

Moreover, Corollary~\ref{cor:reducing_infinite_strips} shows that even infinite sequence of strips glued in a way shown in Figure~\ref{fig:reducing_c3_infinite} can also be replaces with one open strip. 
\end{remark}

Let $\stripS$ be a model strip such that $\partial_{+}\stripS =(-2,2)$ and $\partial_{-}\stripS =(-2,2)$, see Figure~\ref{fig:strip_rectangle}.
\begin{figure}[h]
\includegraphics[height=0.9cm]{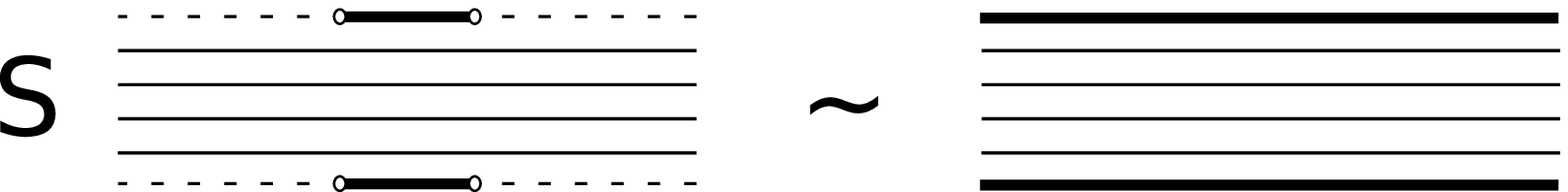}
\caption{}\label{fig:strip_rectangle}
\end{figure}

Define the following two homeomorphisms $\phi_{c}, \phi_{m}: \partial_{+}\stripS \to \partial_{-}\stripS$ by
\begin{align*}
\phi_c(t)&=t, &
\phi_m(t)&=-t,
\end{align*}
for all $t\in(-2,2)$.
Let also $C = \stripS/ \phi_{c}$, $M = \stripS / \phi_{m}$ be the quotients of $\stripS$ obtained by identifying its boundary components via $\phi_c$ and $\phi_m$ respectively.
It follows from Lemma~\ref{lm:reducing_strips} that $C$ is an open cylinder and $M$ is an open M\"obius band.
Moreover, the leaf obtained by gluing $\partial_{+}\stripS$ with $\partial_{-}\stripS$ is of type {\typeCycle}. 

\begin{lemma}\label{lm:reducing_c2}
If $\stripSurf$ has a leaf of type {\typeCycle}, then it is $\FF$-homeomorphic either with $C$ or with $M$.
\end{lemma}
\begin{proof}
Suppose $\omega = \partial_{-}\hstrip_{\cind} = \partial_{+}\hstrip_{\cind}$ is a leaf of type {\typeCycle} for some $\cind\in \Aind$.
Let also \[ (a,b)\times\{-1\} = \partial_{-}\strip_{\cind} \qquad \text{and} \qquad (c,d)\times\{1\} = \partial_{+}\strip_{\cind}\] be the boundary components of $\strip_{\cind}$, so 
\[
\strip_{\cind} \ = \ (a,b)\times\{-1\} \ \bigcup \ \RRR\times(-1,1) \ \bigcup \ (c,d)\times\{1\}.
\]
Since $\omega = \partial_{-}\hstrip_{\cind} = \partial_{+}\hstrip_{\cind}$, it follows that 
\[\strip_{\cind} = \qmap^{-1}(\hstrip_{\cind}) \subset \bigsqcup\limits _{\aind\in \Aind} \strip_{\aind}.\]
But $\strip_{\cind}$ is an open closed subset of $\bigsqcup\limits _{\aind\in \Aind} \strip_{\aind}$, whence from the definition of factor topology on $\stripSurf$ it follows that $\hstrip_{\cind}$ is open closed in $\stripSurf$, and so it coincides with $\stripSurf$.

Thus $\stripSurf$ is obtained by gluing $(c,d)\times\{+1\}$ with $(a,b)\times\{-1\}$ by an affine homeomorphism $\phi_{\cind}$.
If $\phi_{\cind}$ preserves orientation then $\stripSurf$ is homeomorphic with $C$.
Otherwise, $\stripSurf$ is homeomorphic with $M$.
\end{proof}

\begin{definition}
A stripped surface will be called \myemph{reduced} if it has no leaves of types {\typeCycle} and {\typeReduce} that is every leaf of type {\typeGlued} is of type {\typeSpec}.
\end{definition}

\begin{theorem}\label{th:reducing_stripped_surfaces}
Every connected stripped surface $\stripSurf$ with countable base is $\FF$-homeo\-morphic either to a cylinder $C$ or to a M\"obius band $M$, or to a reduced surface.
\end{theorem}
\begin{proof}
Since $\stripSurf$ has countable base, it follows that the number of strips in $\stripSurf$ is at most countable.

We will now define a certain graph $G$ in the following way:
\begin{itemize}
\item the vertices of $G$ are strips of $\stripSurf$ containing leaves of types {\typeCycle} or {\typeReduce};

\item
if there exists a leaf $\omega$ of type {\typeCycle} such that $\omega = \partial_{-}\hstrip_{\cind} = \partial_{+}\hstrip_{\cind}$ for some $\cind\in\Aind$, then we assume that $\omega$ is a loop at vertex $\hstrip_{\cind}$;

\item two such vertices $\hstrip_{\cind}$ and $\hstrip_{\cind'}$ are connected by an edge in $G$ if and only if there exists a leaf $\omega$ of type {\typeReduce} such that $\omega = \partial_{-}\hstrip_{\cind} = \partial_{+}\hstrip_{\cind'}$ or $\omega = \partial_{+}\hstrip_{\cind} = \partial_{-}\hstrip_{\cind'}$.
\end{itemize}
It follows that each vertex of $G$ has degree either $1$ or $2$.
Therefore every connected component of $G$ has one of the forms (i)-(v) shown in Figure~\ref{fig:graph_c3}.
\begin{figure}[h]
\includegraphics[width=0.9\textwidth]{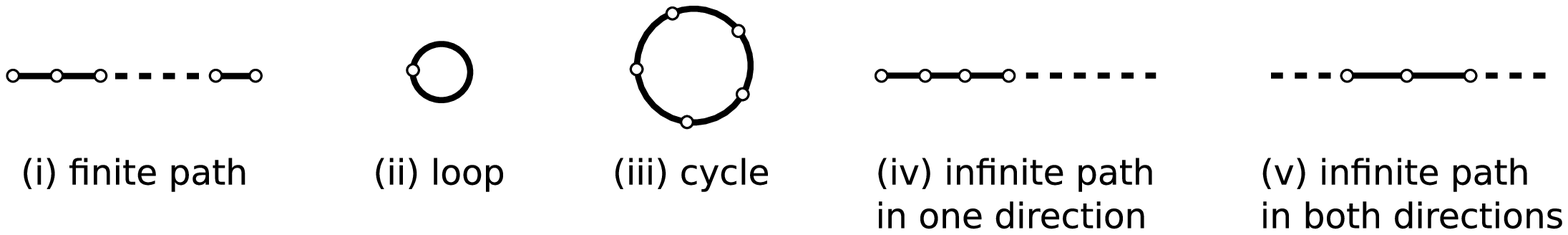}
\caption{}\label{fig:graph_c3}
\end{figure}

Suppose $G$ is non-empty and let $K$ be a connected component of $G$.
Consider the cases (i)-(v) of Figure~\ref{fig:graph_c3}.

\medskip

{\bf (i)}~Suppose $K$ is a finite non-closed path consisting of $n$ edges.
As noted in Remark~\ref{rem:eliminatinc_c3_leaf} each leaf of type {\typeReduce} can be regarded as an internal one after changing partition of $\stripSurf$ by strips without changing foliation $\FZ$. 
This means that each individual edge of $G$ with distinct ends can be reduced.
Applying Lemma~\ref{lm:reducing_strips} $n$ times one can completely reduce all edges of $K$.

\medskip

{\bf (ii)}~Suppose $K$ is a loop, so its edge corresponds to a leaf of type {\typeCycle}.
Then by Lemma~\ref{lm:reducing_c2} $\stripSurf$ is $\FF$-homeomorphic either with $C$ or with $M$.

\medskip

{\bf (iii)}~Suppose $K$ is a cycle of $n$ edges.
Then by arguments similar to (i) the situation reduces to the case (ii).

\medskip

{\bf (iv)}~Suppose $K$ is an infinite in one direction path.
Then the arguments of the case (i) can not be applied as it requires to consider infinite sequence.

Let $\hstrip_0, \hstrip_1, \ldots$ be an infinite sequence of strips in $\stripSurf$ corresponding to vertices of $K$ such that $\omega_i$, $i\geq0$ is a leaf of type {\typeReduce} being an edge between vertices $\hstrip_{i}$ and $\hstrip_{i+1}$ in $K$.
Interchanging $\partial_{-}\strip_i$ with $\partial_{+}\strip_i$ one may assume that 
\[
\omega_i = \partial_{+}\hstrip_{i} = \partial_{-}\hstrip_{i+1}, \qquad i\geq0.
\]

Notice that each $\stripS_i$ is embedded into $\stripSurf$.
Denote $\stripT = \cup_{i=1}^{\infty}\hstrip_i$, so $\hstrip_0$ is omitted.
Then it follows from Corollary~\ref{cor:reducing_infinite_strips} that $\stripT$ is $\FF$-homeomorphic to a model strip having only one boundary component.

\begin{sublemma}
$\omega_0$ splits $\stripSurf$ so that $\stripT\setminus\omega_0$ is a connected component of $\stripSurf\setminus\omega_0$.
\end{sublemma}
\begin{proof}
It suffices to show that $\stripT\setminus\omega_0$ is open and closed in $\stripSurf\setminus\omega_0$.

1) First we will check the \myemph{openness of $\stripT\setminus\omega_0$}.
Let $x\in \stripT\setminus\omega_0$.

a) If $x\in \Int{\hstrip_i} \subset \stripT$ for some $i\geq1$, then $\Int{\hstrip_i}$ is an open neighborhood of $x$ in $\stripT$.

b) If $x\in \omega_i = \partial_{+}\hstrip_{i} = \partial_{-}\hstrip_{i+1}$, $i\geq1$, then $x$ has an open neighborhood $U$ intersecting $\stripS_{i}$ and $\stripS_{i+1}$ only, and so $U\subset\stripT$.

Thus $\stripT\setminus\omega_0$ is open in $\stripSurf\setminus\omega_0$.
 
\medskip

2) Now let us show that \myemph{$\stripT\setminus\omega_0$ is closed in $\stripSurf\setminus\omega_0$}.
Let $\{x_j\}_{j\in\NNN} \subset \stripT\setminus\omega_0$ be a sequence converging to some $x\in\stripSurf\setminus\omega_0$.
We should prove that $x\in \stripT\setminus\omega_0$.
Consider two cases.

a) Suppose $x \in \Int{\hstrip_{\aind}}$ for some $\aind\in\Aind$.
Since $\Int{\hstrip_{\aind}}$ is an open neighborhood of $x$ in $\stripSurf\setminus\omega_0$, it follows that $x_j\in\Int{\hstrip_{\aind}}$ for some $j\in\NNN$.
But $x_j\in \hstrip_i$ for some $i\geq1$, whence $x \in \hstrip_{\aind} = \hstrip_i \subset \stripT$ as well.

b) Suppose $x \in \partial_{-}\hstrip_{\aind}\cap\partial_{+}\hstrip_{\aind'}$ for some $\aind,\aind'\in\Aind$.
Then $x$ has a neighborhood $U$ intersecting $\hstrip_{\aind}$ and $\hstrip_{\aind'}$ only.
Take $x_j\in U$. 
Then $x_j \in \hstrip_{\aind}\cup \hstrip_{\aind'}$ which implies that one of these strips coincides with $\hstrip_i$ for some $i\geq1$.
But $\partial\hstrip_i$ intersects only $\partial\hstrip_{i+1}$ and $\partial\hstrip_{i-1}$ for $i>1$.
Hence $\hstrip_{\aind}$ and $\hstrip_{\aind'}$ are contained in $\stripT$, and so $x\in\stripT$ as well.
\end{proof}

It follows from this lemma that one can replace $\stripT$ with a model strip, and so the situation reduces to the case when $K$ consists of a unique edge with distinct vertices.
Then by the case (i) it can be completely eliminated.

\medskip

{\bf (v)}~Finally suppose that $K$ is an infinite in two directions path.
By arguments of (iv) one can replace each of infinite ends of $K$ with a model strip, and then by (i) eliminate all edges of $K$.
This also implies that $\stripSurf$ is $\FF$-homeomorphic with an open model strip $\RRR\times(-1,1)$.

\medskip

Since each strip $\hstrip_{\aind}$ may correspond to at most one connected component of $G$, it follows that one can apply cases (i)-(v) mutually to each of the connected components of $G$.
This allows to eliminate all leaves of type {\typeReduce} or prove that $\stripSurf$ is $\FF$-homeomorphic either with $C$ or with $M$.
\end{proof}

\section{Homeomorphisms group of $\FZ$}\label{sect:homeomorphisms_of_foliations}
Let $\stripSurf$ be a stripped surface.
Denote by $\HFZ$ the group of $\FF$-homeomorphisms of $\stripSurf$ preserving $\FZ$, that is for each $\dif\in\HFZ$ and each leaf $\omega\in\FZ$ the image $\dif(\omega)$ is a leaf of $\FZ$ as well.

Let also $\HOFZpr$ be the subgroup of $\HFZ$ consisting of homeomorphisms $\dif$ such that $\dif(\omega)=\omega$ for each leaf of $\FZ$ and $\dif$ preserves orientation of $\omega$.

Endow $\HFZ$ with the corresponding compact open topology and let $\HOFZ$ be the identity path component of $\HFZ$, so it consists of homeomorphisms $\dif\in\HFZ$ isotopic in $\HFZ$ to $\id_{\stripSurf}$.

Let $\Sigma(\FZ)$ be the union of leaves of types {\typeBd} and {\typeSpec}.
Evidently, $\dif(\Sigma(\FZ)) = \Sigma(\FZ)$ for each $\dif\in\HFZ$.

First we will consider the case when $\stripSurf$ is a model strip.
\begin{lemma}\label{lm:homeo_of_model_strips}
Let $\strip \subset \RRR\times[-1,1]$ be a model strip and $\gdif\in\HF{\strip}$.
Then 
\begin{equation}\label{equ:formulae_for_h}
\gdif(x,y) = \bigl( \xcoord(x,y), \ \ycoord(y) \bigr),
\end{equation}
where $\ycoord:[-1,1]\to[-1,1]$ is a homeomorphism, and $\xcoord:\strip\to\RRR$ is a continuous function such that for each $y\in(-1,1)$ the correspondence $x \mapsto \xcoord(x,y)$ is a homeomorphism $\RRR\to\RRR$.
\end{lemma}
\begin{proof}
Since $\gdif$ preserves leaves of $\FS$, i.e.\! the lines $\RRR\times y$, $y\in(-1,1)$, it follows that $\ycoord$ does not depend on $x$.
Moreover, as $\gdif$ homeomorphically maps leaves onto leaves, the map $x\mapsto\xcoord(x,y)$ is a homeomorphism $\RRR\times y\to\RRR \times \ycoord(y)$.
Finally, $\ycoord:(-1,1)\to(-1,1)$ is a strictly monotone surjective continuous function, for each $y\in(-1,1)$.
Therefore it extends to a self-homeomorphism of $[-1,1]$.
\end{proof}

The following lemma is easy and we leave it for the reader.
\begin{lemma}\label{lm:lifting_of_strip_homeo}
Let $\hstrip\subset\stripSurf$ be a strip containing leaves from $\Sigma(\FZ)$, $\xi:\strip\to\hstrip$ be the corresponding chart, and $\dif\in\HFZ$.
If $\dif(\hstrip)=\hstrip$, then $\dif$ lifts to a homeomorphism $\gdif:\strip\to\strip$ of the model strip $\strip$ such that $\xi\circ\gdif = \dif\circ\xi$.
\qed
\end{lemma}

\begin{lemma}\label{lm:HOFZpr_is_contractible}
The group $\HOFZpr$ is contractible.
\end{lemma}
\begin{proof}
Let $\xi:\strip\to\hstrip$ be a strip of $\stripSurf$ and $\dif\in\HOFZpr$.
By assumption $\dif(\hstrip)=\hstrip$, whence by Lemma~\ref{lm:lifting_of_strip_homeo} $\dif|_{\hstrip}$ lifts to a self-homeomorphism $\gdif:\strip\to\strip$ such that $\xi\circ\gdif=\dif\circ\xi$.
Moreover, by Lemma~\ref{lm:homeo_of_model_strips} and assumption that $\dif$ preserves leaves with their orientations we have that
\[
\gdif(x,y) =  \bigl(\xcoord(x,y), y \bigr),
\]
where the correspondence $x \mapsto \xcoord(x,y)$, $y\in[-1,1]$, is a self homeomorphism of $\RRR$ preserving orientation.

Then an isotopy $G:\strip \times [0,1] \to \strip$ between $\gdif$ and $\id_{\strip}$ can be defined by the formula:
\begin{equation}\label{equ:contration_Hpr_on_strip}
G(x,y; t) = \bigl( (1-t)\xcoord(x,y)+tx, \ y \bigr)
\end{equation}
for $(x,y)\in\strip\subset\RRR\times[-1,1]$.

We will show that formulas for $G$ on distinct strips agree with \textit{affine} gluing maps $\phi_{\bind}$ for all $\bind\in \Bind$.
More precisely, let $\xi:\strip\to\hstrip$ and $\xi:\strip'\to\hstrip'$ be two strips of $\stripSurf$.
For the convenience of notation assume that 
\begin{align*}
\bdY & \ \subset \ \partial_{+}\strip \ \subset \ \RRR\times\{1\}, 
&
\bdX & \ \subset \ \partial_{-}\strip' \ \subset \ \RRR\times\{-1\},
\end{align*}
be two boundary components glued via an affine homeomorphism $\phi:\bdY \to \bdX$.
Let also $\gdif:\strip\to\strip$ and $\gdif':\strip'\to\strip'$ be the corresponding liftings of $\dif|_{\hstrip}$ and $\dif|_{\hstrip'}$ respectively, and $G$ and $G'$ the corresponding isotopies for $\gdif$ and $\gdif'$ given by~\eqref{equ:contration_Hpr_on_strip}.
We have to prove that for each $t\in[0,1]$ the following commutative diagram holds true:
\[
\begin{CD}
\bdY @>{G_t}>> \bdY  \\
@V{\phi}VV @VV{\phi}V \\
\bdX @>{G'_t}>> \bdX
\end{CD}
\]
This diagram trivially holds for $t=1$ when $G_1$ and $G'_1$ are identity maps.
Moreover it also holds for $t=0$ when $G_0=\gdif$ and $G'_0=\gdif'$, since $\dif$ agrees on $\xi(\bdY)=\xi'(\bdX)$, so
\begin{equation}\label{equ:diagram_for_t0}
\phi \circ \gdif = \gdif' \circ \phi: \bdY\longrightarrow\bdX.
\end{equation}

We can assume that 
\begin{align*}
\phi(s,1) &= (u s + v, -1), &
\gdif(x,1) &= \bigl( \xcoord(x,1), 1 \bigr)&
\gdif'(x,-1) &= \bigl( \xcoord'(x,-1), -1 \bigr)
\end{align*}
for some $u,v\in\RRR$, and continuous functions $\xcoord:\bdY\to\RRR$ and $\xcoord':\bdX\to\RRR$.
For simplicity we will also omit the second coordinate $\pm1$.
Then 
\begin{align*}
\phi(s) &= u s + v, &
\gdif|_{\bdY}(x) &= \xcoord(x), &
\gdif'|_{\bdX}(x) &= \xcoord'(x),
\end{align*}
and so we get from~\eqref{equ:diagram_for_t0} that 
\[
\phi \circ \gdif|_{\bdY}(x) = u \xcoord(x) + v =  \xcoord'(ux+v) = \gdif'|_{\bdX} \circ \phi.
\]
Hence 
\begin{align*}
\phi \circ G_t(x)
 &= u ( (1-t)\xcoord(x) + tx ) + v  \\
 &= (1-t) \bigl(u \xcoord(x) + v\bigr) + t(ux+v) \\ 
 &=  (1-t)\xcoord'(ux + v) + t(ux+v) \\
 &= G'_t \circ \phi(x).
\end{align*}
Thus isotopies~\eqref{equ:contration_Hpr_on_strip} agree on distinct model strips, and therefore they yield a unique isotopy between $\dif$ and $\id_{\stripSurf}$ in $\HOFZpr$.
One can easily check that such an isotopy is continuous in $(\dif,t)\in\HOFZpr\times I$ and so it yields a contraction of $\HOFZpr$.
\end{proof}

\begin{theorem}\label{th:characterzation_of_H0}
Let $\stripSurf$ be a connected reduced stripped surface and $\dif\in\HFZ$.
Then $\dif\in\HOFZ$ if and only if the following three conditions hold:
\begin{itemize}
\item[\rm(a)]
$\dif(\hstrip_{\aind})=\hstrip_{\aind}$ for all $\aind\in\Aind$;
\item[\rm(b)]
if $\gdif_{\aind}:\strip_{\aind}\to\strip_{\aind}$, $\gdif(x,y) = \bigl( \xcoord(x,y), \ \ycoord(y) \bigr)$, $\aind\in\Aind$, is a unique lifting of $\dif|_{\hstrip_{\aind}}$ given by~\eqref{equ:formulae_for_h}, then $\ycoord$ is increasing and $\xcoord(x,y)$ also increasing for each fixed $y\in(-1,1)$.
\item[\rm(c)]
$\dif$ leaves invariant each leaf $\omega \subset \Sigma(\FZ)$ and preserves its orientation.
\end{itemize}
Moreover, $\HOFZpr$ is a strong deformation retract of $\HOFZ$, and in particular, $\HOFZ$ is contractible as well.
\end{theorem}
\begin{proof}
Let $Q \subset \HFZ$ be a subgroup consisting of maps satisfying (a), (b), (c).
We should prove that $\HOFZ = Q$.

\medskip

{\em Inclusion $\HOFZ \subset Q$.}
Suppose $\dif\in\HOFZ$, so there exists an isotopy $\dif_t:\stripSurf\to\stripSurf$, $t\in I$, such that $\dif_0=\id_{\stripSurf}$, $\dif_1= \dif$, and $\dif_t\in\HFZ$ for all $t\in I$.

(c) Since $\dif_t(\Sigma(\FZ)) = \Sigma(\FZ)$ and every leaf $\omega \subset \Sigma(\FZ)$ is a \myemph{path component} of $\Sigma(\FZ)$, it follows that
\[ \dif(\omega) = \dif_1(\omega)=\dif_t(\omega)=\dif_0(\omega) = \omega, \qquad t\in I.\]
Moreover, the restriction of $\dif_t|_{\omega}:\omega \to \omega$ is an isotopy between $\id_{\omega}$ and $\dif|_{\omega}$, whence $\dif|_{\omega}$ preserves orientation.

\medskip

(a) It follows further, that $\dif$ also leaves invariant every connected components of $\stripSurf\setminus\Sigma(\FZ)$.
But every such component is the interior of some strip, whence $\dif(\hstrip_{\aind})=\hstrip_{\aind}$ for each $\aind\in\Aind$, which completes (a).

\medskip

(b) Moreover, let $(\gdif_{\aind})_t:\Int{\strip_{\aind}}\to\Int{\strip_{\aind}}$ be a lifting of $\dif_t|_{\Int{\hstrip_{\aind}}}$.
Then $(\gdif_{\aind})_t = \bigl( \xcoord_t(x,y), \ \ycoord_t(y) \bigr)$, where $\xcoord_t$ and $\ycoord_t$ are continuous in $(x,y,t)$.
Moreover, $\{\ycoord_t\}_{t\in I}$ is an isotopy between $\ycoord_0 =\id_{(-1,1)}$ and $\ycoord_1=\ycoord$.
Hence $\ycoord$ is increasing.

Similarly, for each fixed $y\in(-1,1)$ the correspondence $x\mapsto \xcoord_t(x,y)$ is also an isotopy between $\id_{\RRR}$ and $x\mapsto\xcoord(x,y)$.
Hence the latter is also an increasing function.

\medskip

To prove the inverse inclusion $\HOFZ \supset Q$ we need the following lemma:
\begin{sublemma}
$\HOFZpr$ is a strong deformation retract of $Q$.
\end{sublemma}
\begin{proof}
Let $\dif\in Q$, $\aind\in\Aind$, and $\gdif_{\aind}:\strip_{\aind}\to\strip_{\aind}$ be a lifting of $\dif$, see Lemma~\ref{lm:lifting_of_strip_homeo}.
So by (c)
\[ \gdif_{\aind}(x,y) = \bigl( \xcoord_{\aind}(x,y), \ \ycoord_{\aind}(y) \bigr), \]
where then $\ycoord$ is increasing and $\xcoord_{\aind}(x,y)$ is increasing in $x$ for each fixed $y\in(-1,1)$.
Define an isotopy $F_{\aind}:\strip_{\aind}\times I\to \strip_{\aind}$ by
\begin{equation}\label{equ:deformation_of_Q}
F_{\aind}(x,y,t) = \bigl( \xcoord_{\aind}(x,y), ty + (1-t) \ycoord_{\aind}(y) \bigr).
\end{equation}
Then $F_{\aind}$ is fixed on $\partial\strip_{\aind}$, $(F_{\aind})_0 = \gdif_{\aind}$ and $(F_{\aind})_1$ preserves each leaf of $\FCanon{\strip_{\aind}}$ with its orientation.
Hence the family of isotopies $\{F_{\aind}\}_{\aind\in\Aind}$ yield an isotopy $F(\dif): \stripSurf\times I \to I$ of $\dif$ to a diffeomorphism which preserves each leaf of $\FZ$ with its orientation.
In other words, $F(\dif)_0=\dif$, $F(\dif)_t\in Q$ for all $t\in I$, and $F(\dif)_1 \in \HOFZpr$.

One can easily check that the map $F:Q\times I \to Q$ given by $F(\dif,t) = F(\dif)_t$ is continuous.
Moreover, if $\dif\in\HOFZpr$, then in~\eqref{equ:deformation_of_Q} $\ycoord_{\aind}(y)=y$, whence $F_{\aind}(x,y,t)=\gdif_{\aind}(x,y)$, and so $F(\dif,t) = \dif$ for all $t\in I$.
In other words, $F$ is a deformation fixed on $\HOFZpr$, whence $\HOFZpr$ is a strong deformation retract of $Q$.
\end{proof}

Since $\HOFZpr$ is connected (even contractible), it follows from this lemma that $Q$ is also connected.
But $\id_{\stripSurf}\in Q$, whence $Q \subset \HOFZ$.
Theorem~\ref{th:characterzation_of_H0} completed.
\end{proof}

\def\cprime{$'$} \def\cprime{$'$} \def\cprime{$'$} \def\cprime{$'$}
\providecommand{\bysame}{\leavevmode\hbox to3em{\hrulefill}\thinspace}
\providecommand{\MR}{\relax\ifhmode\unskip\space\fi MR }
\providecommand{\MRhref}[2]{%
  \href{http://www.ams.org/mathscinet-getitem?mr=#1}{#2}
}
\providecommand{\href}[2]{#2}


\begin{thebibliography}{10}

\bibitem{BolsinovFomenko:1997}
A.~V. Bolsinov and A.~T. Fomenko, \emph{Vvedenie v topologiyu integriruemykh
  gamiltonovykh sistem ({I}ntroduction to the topology of integrable
  hamiltonian systems)}, ``Nauka'', Moscow, 1997 (Russian). \MR{MR1664068
  (2000g:37079)}

\bibitem{Boothby:AJM_1:1951}
William~M. Boothby, \emph{The topology of regular curve families with multiple
  saddle points}, Amer. J. Math. \textbf{73} (1951), 405--438. \MR{0042692
  (13,149b)}

\bibitem{Boothby:AJM_2:1951}
\bysame, \emph{The topology of the level curves of harmonic functions with
  critical points}, Amer. J. Math. \textbf{73} (1951), 512--538. \MR{0043456
  (13,266a)}

\bibitem{JenkinsMorse:AJM:1952}
James Jenkins and Marston Morse, \emph{Contour equivalent pseudoharmonic
  functions and pseudoconjugates}, Amer. J. Math. \textbf{74} (1952), 23--51.
  \MR{0048642 (14,46b)}

\bibitem{JenkinsMorse:APNASUSA:1953}
\bysame, \emph{Conjugate nets, conformal structure, and interior
  transformations on open {R}iemann surfaces}, Proc. Nat. Acad. Sci. U. S. A.
  \textbf{39} (1953), 1261--1268. \MR{0058724 (15,415e)}

\bibitem{JenkinsMorse:ActaMath:1954}
\bysame, \emph{Curve families {$F^*$} locally the level curves of a
  pseudoharmonic function}, Acta Math. \textbf{91} (1954), 1--42. \MR{0062292
  (15,956h)}

\bibitem{JenkinsMorse:UMichPress:1955}
\bysame, \emph{Conjugate nets on an open {R}iemann surface}, Lectures on
  functions of a complex variable, The University of Michigan Press, Ann Arbor,
  1955, pp.~123--185. \MR{0069900 (16,1097b)}

\bibitem{Kamke:MA:1928}
E.~Kamke, \emph{Zur {T}heorie der {D}ifferentialgleichungen}, Math. Ann.
  \textbf{99} (1928), no.~1, 602--615. \MR{1512468}

\bibitem{Kaplan:DJM:1940}
Wilfred Kaplan, \emph{Regular curve-families filling the plane, {I}}, Duke
  Math. J. \textbf{7} (1940), 154--185. \MR{0004116 (2,322c)}

\bibitem{Kaplan:DJM:1941}
\bysame, \emph{Regular curve-families filling the plane, {II}}, Duke Math J.
  \textbf{8} (1941), 11--46. \MR{0004117 (2,322d)}

\bibitem{Kaplan:TrAMS:1948}
\bysame, \emph{Topology of level curves of harmonic functions}, Trans. Amer.
  Math. Soc. \textbf{63} (1948), 514--522. \MR{0025159 (9,606f)}

\bibitem{Morse:JMPA:1956}
M.~Morse, \emph{La construction topologique d'un r\'eseau isotherme sur une
  surface ouverte}, J. Math. Pures Appl. (9) \textbf{35} (1956), 67--75.
  \MR{0077648 (17,1071d)}

\bibitem{Morse:FM:1952}
Marston Morse, \emph{The existence of pseudoconjugates on {R}iemann surfaces},
  Fund. Math. \textbf{39} (1952), 269--287 (1953). \MR{0057338 (15,210a)}

\bibitem{Oshemkov:PSIM:1995}
A.~A. Oshemkov, \emph{{M}orse functions on two-dimensional surfaces. {C}oding
  of singularities}, Trudy Mat. Inst. Steklov. \textbf{205} (1994), no.~Novye
  Rezult. v Teor. Topol. Klassif. Integr. Sistem, 131--140. \MR{MR1428674
  (97m:57045)}

\bibitem{Polulyakh:UMZ:ENG:2015}
Eugene Polulyakh, \emph{Kronrod-reeb graphs of functions on non-compact
  surfaces}, Ukrainian Math. Journal \textbf{67} (2015), no.~3, 375--396
  (Russian).

\bibitem{PolulyakhYurchuk:ProcIM:2009}
Eugene Polulyakh and Iryna Yurchuk, \emph{On the pseudo-harmonic functions
  defined on a disk}, Pr. Inst. Mat. Nats. Akad. Nauk Ukr. Mat. Zastos.
  \textbf{80} (2009), 151 (Ukrainian).

\bibitem{Sharko:UMZ:2003}
V.~V. Sharko, \emph{Smooth and topological equivalence of functions on
  surfaces}, Ukra\"\i n. Mat. Zh. \textbf{55} (2003), no.~5, 687--700.
  \MR{MR2071708 (2005f:58075)}

\bibitem{Sharko:Zb:2006}
\bysame, \emph{Smooth functions on non-compact surfaces}, Pr. Inst. Mat. Nats.
  Akad. Nauk Ukr. Mat. Zastos. \textbf{3} (2006), no.~3, 443--473,
  arXiv:math/0709.2511.

\bibitem{SharkoSoroka:MFAT:2015}
V.~V. Sharko and Yu.~Yu. Soroka, \emph{Topological equivalence to a
  projection}, Methods Funct. Anal. Topology \textbf{21} (2015), no.~1, 3--5.
  \MR{3407916}

\bibitem{Stoilov:1964}
S.~Stoilov, \emph{Lectures on topological principles of the theory of analytic
  functions}, Translated from the French by E. T. Ste\v ckina. With a foreword
  by B. V . \u Sabat, Izdat. ``Nauka'', Moscow, 1964 (Russian). \MR{0188461 (32
  \#5899)}

\bibitem{Weaver:AMSCP:1942}
G.~T. Whyburn, \emph{Analytic topology}, Amer. Math. Soc. Colloquium
  Publications \textbf{28} (1942).

\end{thebibliography}
\end{document}